\documentclass[12pt]{amsart}
\usepackage{amsmath,amssymb,amsfonts,enumerate}

\newcommand{\inv}{^{\raisebox{.2ex}{$\scriptscriptstyle-1$}}}   

\newtheorem{theorem}{Theorem}[section]
\newtheorem{lemma}[theorem]{Lemma}

\newtheorem{corollary}[theorem]{Corollary}
\newtheorem{proposition}[theorem]{Proposition}
\theoremstyle{definition}

\theoremstyle{definition}

\begin{document}

\title[Primitive Hyperideals and Hyperstructure spaces ]{Primitive Hyperideals and Hyperstructure spaces of hyperrings}

\author[Bijan Davvaz, Amartya Goswami, Karin-Therese Howell]{Bijan Davvaz$^1$, Amartya Goswami$^{2,*}$, and Karin-Therese Howell$^{3,*}$}

\address{$^{1}$Department
of Mathematical Sciences, Yazd University, Yazd, Iran.}

\address{$^{2}$Department of Mathematics and Applied Mathematics, University of Johannesburg, Auckland Park, 2006, South Africa.}

\address{$^{3}$Department of Mathematical Sciences, Stellenbosch University, Stellenbosch 7600, South Africa.}
\address{
$^{*}$National Institute for Theoretical and Computational Sciences (NITheCS), South Africa.}

\email{davvaz@yazd.ac.ir}
\email{agoswami@uj.ac.za}
\email{kthowell@sun.ac.za}



\begin{abstract}
We introduce primitive hyperideals of a hyperring $R$ and show relations with $R$ itself, and with maximal and prime hyperideals of $R$. We endow a Jacobson topology on the set of primitive hyperideals of $R$ and study topological properties of the corresponding hyperstructure space.
\end{abstract}

\maketitle
\section{Introduction}

The notion of \emph{multi-valued} algebraic structures were first considered in \cite{M34}, where \emph{hypergroups} have been introduced, which is a generalization over groups by allowing the binary operation to be multi-valued. Later, in \cite{K47}, the concept of \emph{hyperrings} have been introduced. Since its first appearance, (mostly commutative) hyperrings have been intensively studied both in algebraic as well as in geometric contexts. In  \cite{DavLeo} (see also \cite{DS06}), a comprehensive account on various algebraic properties of hyperrings (as well as their generalizations) can be found. For applications of hyperrings in geometry, we refer to \cite{CC10, CC11,CL-F03,J18}. As far as noncommutative hyperrings are concerned not much structural study has been done so far.

It is well known (see \cite{J45} and \cite{J56} ) that for noncommutative rings, the notion of primitive ideals play a crucial role in determining its structure theory. Furthermore, in \cite{J56}, a hull-kernel-type topology has been endowed on the set of all primitive ideals of a ring, and representations of biregular rings were studied.  It is also worth it to mention that primitive ideals have shown their immense importance in understanding structural aspects of modules \cite{J56, R88}, Lie algebras \cite{KPP12}, enveloping algebras \cite{D96,J83}, PI-algebras \cite{J75}, quantum groups \cite{J95}, skew polynomial rings \cite{I79}, and others.

The aim of this paper is to  introduce primitive hyperideals of a (Krasner) hyperring and studying some of their properties. We show implications between prime, maximal, and primitive hyperideals of a hyperring. We also characterize simple hypermodules. Similar to \cite{J56}, we impose Jacobson topology on the set of primitive hyperideals of a hyperring and investigate topological properties of the corresponding hyperstructure space. We characterize irreducible closed subsets of a hyperstructure space and prove every irreducible closed subset of a hyperstructure  space has a unique generic point. We give sufficient condition for the space to be Noetherian. We study continuous maps between such spaces.

\section{Preliminaries} 

Suppose $R$ is a nonempty set and $\mathcal{P}^*(R)$ is the set of all nonempty subsets of $R$. 
A  \emph{Krasner hyperring} is a system $(R,+,\cdot,-,0)$ such that

(I) $(R,+,0)$ is a \emph{canonical hypergroup}, that is, $+\colon R\times R\to \mathcal{P}^*(R)$ is a hyperoperation on $A$ satisfying the following properties for all $a,b,c \in R$:
\begin{enumerate}[\upshape(i)]
	
\item $a+b=b+a;$
	
\item $a+(b+c)=(a+b)+c$;
	
\item  there exists $0\in A$ such that $a+0=\lbrace a\rbrace$;
	
\item  for every $a$, there exists a unique $-a\in A$ such that $0\in a-a;$
	
\item if $a\in b+c$, then $c\in -b+a$ and $a \in c-b$,
\end{enumerate}

(II) $(R, \cdot)$ is a semigroup,

(III) $a \cdot 0 = 0 \cdot a = 0,$ and

(IV) $a\cdot (b+c)=  a\cdot b + a \cdot c,$

(V) $(a+b)\cdot c=a\cdot c + b\cdot c,$

for all $a,b,c \in R.$

A hyperring $R$ is called \emph{unital} if $R$ has a multiplicative identity, that is, there exists $1\in R$ such that $a\cdot 1=a=1\cdot a$ for all $a\in R$. For simplicity, we shall write $a\cdot b$ as $ab$.
We will restrict our focus to Krasner hyperrings in this paper, so if we refer to a hyperring, it will be a Krasner hyperring. 

A nonempty subset $S$ of a hyperring $R$ is said to be a \emph{subhyperring} of $R$ if $(S,+,\cdot)$ is itself a hyperring. A subhypergroup $\mathfrak{a}$ of a hyperring $R$ is called a \emph{left (right) hyperideal} of $R$ if $r\cdot a \in \mathfrak{a} \,(a\cdot r \in \mathfrak{a})$ for all $r\in R, a \in \mathfrak{a}.$  If $\mathfrak{a}$ is both a left and right hyperideal then $\mathfrak{a}$ is called a \emph{two-sided hyperideal} or simply a \emph{hyperideal}. Unless otherwise stated, we assume all hyperideals are two-sided. If $\mathfrak{a}$ is a hyperideal of $R$, then we can form the \emph{quotient} hyperring $R/\mathfrak{a} = \{ \mathfrak{a}+r\mid r \in R\}$ with the following two operations:
\begin{align*}
(\mathfrak{a}+r_1) + (\mathfrak{a}+r_2) &= \{\mathfrak{a}+r \mid r \in r_1 + r_2\};\\
(\mathfrak{a}+r_1)( \mathfrak{a}+r_2) &= \mathfrak{a}+r_1r_2.
\end{align*}
The following result is known, but the proof is included for completeness.

\begin{proposition}
If $\{\mathfrak{a}_{\lambda}\}_{\lambda \in \Lambda}$ is a nonempty family of hyperideals of a hyperring $R,$ then the following are also hyperideals of $R$.
\begin{enumerate}[\upshape(i)] 
\item $\bigcap_{\lambda \in \Lambda}\mathfrak{a}_{\lambda}$
		
\item $\sum_{\lambda \in \Lambda}\mathfrak{a}_{\lambda} = \{x \mid x \in \sum_{\lambda \in \Lambda}a_{\lambda}, a_{\lambda} \in \mathfrak{a}_{\lambda}\}$
\end{enumerate}
\end{proposition}

\begin{proof}
(i) Suppose that $x,y \in \bigcap_{\lambda \in \Lambda}\mathfrak{a}_{\lambda}.$ Then $x,y \in \mathfrak{a}_{\lambda}$ for all $\lambda \in \Lambda.$ Since each $\mathfrak{a}_{\lambda}$ is a hyperideal, it follows that $x-y \in \mathfrak{a}_{\lambda}$ for all $\lambda \in \Lambda.$ This implies that $x -y \subseteq \bigcap_{\lambda \in \Lambda}\mathfrak{a}_{\lambda}.$
Now let $r \in R.$ For each $\lambda \in \Lambda,$ since $\mathfrak{a}_{\lambda}$ is a hyperideal of $R,$ it follows that $rx \in \mathfrak{a}_{\lambda}$ and $xr \in \mathfrak{a}_{\lambda},$ and hence we conclude that $rx \in \bigcap_{\lambda \in \Lambda}\mathfrak{a}_{\lambda}$ and $xr \in \bigcap_{\lambda \in \Lambda}\mathfrak{a}_{\lambda}.$ 
	
(ii) Suppose that $x,y \in \sum_{\lambda \in \Lambda}\mathfrak{a}_{\lambda}.$ Then $x \in \sum_{\lambda\in \Lambda}{a}_{\lambda}$ for some $a_{\lambda} \in \mathfrak{a}_{\lambda}$ and $y \in \sum_{\lambda\in \Lambda}{b}_{\lambda}$ for some $b_{\lambda} \in \mathfrak{a}_{\lambda}.$ Since each $\mathfrak{a}_{\lambda}$ is a hyperideal, it follows that $a_{\lambda}-b_{\lambda} \subseteq \mathfrak{a}_{\lambda}$ for $\lambda \in \Lambda.$ This implies that $x -y \subseteq \sum_{\lambda \in \Lambda}(a_{\lambda}-b_{\lambda}),$ where $a_{\lambda} - b_{\lambda} \subseteq \mathfrak{a}_{\lambda},$ so $x-y \subseteq \sum_{\lambda \in \Lambda}\mathfrak{a}_{\lambda}.$ Now let $r \in R.$ For each $\lambda \in \Lambda,$ since $\mathfrak{a}_{\lambda}$ is a hyperideal of $R,$ it follows that $ra_{\lambda} \in \mathfrak{a}_{\lambda}$ and $a_{\lambda}r \in \mathfrak{a}_{\lambda},$ for each $\lambda \in \Lambda$ and hence we conclude that $rx \in \sum_{\lambda \in \Lambda}ra_{\lambda}$ and $xr \in \sum_{\lambda \in \Lambda}a_{\lambda}r.$ 
\end{proof}

Recall that if $\mathfrak{a}$ and $\mathfrak{b}$ are nonempty subsets of a hyperring $R.$ Then the product $\mathfrak{a}\mathfrak{b}$ is defined by
$$\mathfrak{a}\mathfrak{b}= \left\{x \mid x \in \sum_{i=1}^{n} a_ib_i,a_i\in \mathfrak{a},b_i \in \mathfrak{b}, n \in \mathbb{Z}^{+} \right\} .$$ Moreover, if $\mathfrak{a}$ and $\mathfrak{b}$ are hyperideals, $\mathfrak{a}\mathfrak{b}$ is also a hyperideal of $R$ (see \cite[p. 87]{DavLeo}). 
Let $X$ be a subset of a hyperring $R.$ Let $\{ \mathfrak{a}_{i}\mid i \in I\}$ be the family of all hyperideals in $R$ which contain $X$. Then $\bigcap_{i\in I} \mathfrak{a}_i$, is called the hyperideal \emph{generated by $X$} and we denoted it by $\langle X \rangle$. 
A proper hyperideal $\mathfrak{m}$ of a hyperring $R$ is called \emph{maximal} if the only hyperideals of $R$ that contain $\mathfrak{m}$ are $\mathfrak{m}$ itself and $R.$ 
A proper hyperideal $\mathfrak{p}$ of a hyperring $R$ is called \emph{prime} if for every pair of hyperideals $\mathfrak{a}$ and $\mathfrak{b}$ of $R$ 
$\mathfrak{a}\mathfrak{b} \subseteq \mathfrak{p}$ implies either $\mathfrak{a}\subseteq \mathfrak{p}$ or $\mathfrak{b}\subseteq \mathfrak{p}.$ 

\begin{lemma}\label{max}
Every proper right hyperideal $\mathfrak{a}$  of a unital hyperring $R$ is contained in a right maximal hyperideal of $R$.
\end{lemma}

\begin{proof}
Suppose $\mathcal{U} = \{ \mathfrak{u}\mid \mathfrak{u}\supseteq \mathfrak{a}, \mathfrak{u}\;\text{is a proper hyperideal of}\; R\}$.  Since $\mathfrak{a}\in \mathcal{U},$ the set $\mathcal{U}$ is nonempty. Consider a chain $\{\mathfrak{c}_{\lambda}\}_{\lambda \in \Lambda}$ in $\mathcal{U}.$ Then $\mathfrak{c}=\bigcup_{\lambda \in \Lambda}\mathfrak{c}_{\lambda}$ is a proper hyperideal of $R$ and is an upper bound of the chain $\{\mathfrak{c}_{\lambda}\}_{\lambda \in \Lambda}$. Moreover, $\mathfrak{c}\neq R$ because $1\notin \mathfrak{c}.$ Hence by Zorn's lemma $\mathcal{U}$ contains a maximal element $\mathfrak{m}$, which is  a maximal hyperideal of $R$ containing $\mathfrak{a}.$
\end{proof}

\section{Primitive hyperideals}

Like in rings, to define primitive hyperideals of a hyperring, we require the notion of simple hypermodules. In the next subsection we first study simple hypermodules and their annihilators.

\subsection{Simple hypermodules}
Recall from \cite{M88} that
a (\emph{right}) \emph{Krasner $R$-hypermodule} $M$  is a canonical hypergroup $M$ endowed with an external composition $M \times R \to M$ (defined by $(m,r)\mapsto mr$) satisfying the conditions:
\begin{enumerate}[\upshape(i)]
\item $(m+m') r = m r+m'r$;
	
\item $m (r + r') = m r+m  r'$;
	
\item $m (r r') = (mr) r'$;
	
\item $m0=0;$
\end{enumerate}
for all $m,m' \in M$ and $r,r' \in R.$  If, moreover, $R$ has a multiplicative identity $1$ and $m1 = m$ for all $m\in M,$ then
$M$ is called \emph{unital}.
We shall only consider right Krasner $R$-hypermodules and hence from now on we drop the adjective ``right Krasner'' and simply say $R$-hypermodule. If an $R$-hypermodule $M$ is generated by a single element $m$ of $M$, then $M$ is called \emph{cyclic}, and we denote it by $\langle m \rangle$ or $Rm.$
The proof of the following property of an $R$-hypermodule can be found in \cite{S07}.
\medskip

\begin{lemma}\label{modprop}
If $M$ is an $R$-\emph{hypermodule}  then
$(-m)r=-(mr)=m(-r)$ for all $r\in R$ and $m\in M.$
\end{lemma}

A \emph{subhypermodule} $S$ of a hypermodule $M$  is a subcanonical hypergroup of $M$ such that $sr\subseteq S,$ for all $r\in R$ and for all $s\in S.$ 
If $M$, $N$ are $R$-hypermodules, then a (\emph{strong})\emph{ $R$-hypermodule homomorphism} from $M$ into $N$ is a map $\mu\colon M\to N$ such that $\mu(m+m')=\mu(m)+\mu(m')$ and $\mu(mr)=\mu(m)r$ for all $r\in R$ and for all $m, m'\in M.$ A hypermodule homomorphism $\mu$ is called an \emph{isomorphism} if $\mu$ is also a bijection on the underlying sets. 

If $M$ is a $R$-hypermodule and $K$ is a subhypermodule of $M$, then the set $M/K = \{K+a \mid a \in M\}$ endowed with a hyperoperation  $+ : M/K \times M/K \rightarrow {\mathcal{P}}^{*}(M/K)$
and  an $R$-action $\cdot:M/K \times R \rightarrow M/K$ respectively defined as: 
\begin{align*}
(K+a) + (K+a') &= \{K+b \mid b\in a + a'\};\\ (K+a)\cdot r &= \{K+b \mid b \in a r\},
\end{align*}
for every $a, a',b \in  M$ and $r \in R,$  is called the \emph{quotient hypermodule} of M. It is easy to show (see \cite[Corollary 2.2.8]{S07}) that $\mathtt{ker}(\mu)$ is a subhypermodule of $M$ and $\mathtt{im}(\mu)$ is a subhypermodule of $N.$ 
As for modules over rings, we also have the fundamental theorem of homomorphisms for hypermodules.

\begin{proposition}\cite[Theorem 2.2.14]{S07}\label{fth} If $\mu\colon M\to M'$ is a hypermodule homomorphism, then $M/\mathtt{ker}(\mu)$ is isomorphic to $\mathtt{im}(\mu).$  \end{proposition}

An $R$-hypermodule $M$ is called \emph{simple} if $RM\neq 0$ and  $M$ has no subhypermodules other
than $0$ and $M.$  The following proposition characterizes a simple hypermodule as a cyclic hypermodule generated by a nonzero element.

\begin{proposition}
A nonzero $R$-hypermodule $M$ is simple if and only if $M=mR$ for every nonzero $m\in M$.
\end{proposition} 

\begin{proof}
If $M$ is simple, and then there exists a $0 \neq m \in M$ and $mR$ is a nonzero subhypermodule of $M,$ so $mR = M.$
Conversely, if $N \neq 0$ is a subhypermodule of $M,$ then $N$ must contain a nonzero element, say $m$ of $M$. Then we have that $M =mR \subseteq N,$ showing that $N=M.$
\end{proof}

The following example of subhypermodule is going to play an important role in studying properties of primitive hyperideals.

\begin{lemma}
If $M$ is an $R$-hypermodule and $\mathfrak{a}$ a hyperideal of $R$, then $$M\mathfrak{a}  = \left\{\sum_{i=1}^{k}m_ia_i\mid m_i \in M, a_i \in \mathfrak{a}, k\in \mathbb{Z}^+\right\}$$ is a subhypermodule of $M$.
\end{lemma}

\begin{proof}
Let $\sum_{i=1}^{k}m_ia_i$ and $\sum_{j=1}^{l}m_ja_j$ be two elements of $M\mathfrak{a}.$
Then 
\begin{flalign*}
\sum_{i=1}^{k}m_ia_i - \sum_{j=i}^{l}m_ja_j & = \sum_{i=1}^{k}m_ia_i + \sum_{j=1}^{l}(-m_j)a_j
\end{flalign*}
where $-m_j \in M$ since $(M,+)$ is a canonical hypergroup. Hence, $\sum_{i=1}^{k}m_ia_i - \sum_{j=1}^{l}m_ja_j \subseteq M\mathfrak{a}.$
Now let $r \in R.$ Then 
\begin{flalign*}
\left(\sum_{i=1}^{k}m_ia_i\right)r& = \sum_{i=1}^{k}m_i(a_ir) 
\end{flalign*}
where $a_ir \in R$ since $\mathfrak{a}$ is a hyperideal of $R.$ Thus $\left(\sum_{i=1}^{k}m_ia_i \right)r \in M\mathfrak{a}.$
\end{proof}

If $M$ is an $R$-hypermodule then the additive subhypergroup $Mr$ of $M$ generated by
the elements of the form $\{mr \mid m \in M, r \in R\}$ is a subhypermodule of $M$. 
The (\emph{right}) \emph{annihilator} of a $R$-hypermodule $M$ is defined by $$\mathtt{Ann}_{R}(M)=\{ r\in R\mid mr=0\;\text{for all}\;\; m\in M\}.$$ When $M=\{m\},$ we write $ \mathtt{Ann}_{R}(m)$ for $ \mathtt{Ann}_{R}(\{m\}).$ 
If $\mathtt{Ann}_{R}(M)=\{0\}$ then $M$ is said to be a \emph{faithful} $R$-hypermodule. Like in rings, we have the following. 

\begin{lemma}
An annihilator $\mathtt{Ann}_{R}(M)$ is a hyperideal of $R$.
\end{lemma}

\begin{proof}
Let $x, x' \in \mathtt{Ann}_{R}(M)$, $r\in R$, and $m \in M.$ Then
$$m(x-x')  =mx +m (-x')=mx - mx' =0+0 = 0,
$$
where  the second equality follows from  Lemma \ref{modprop}.
Furthermore,
$m(xr)  =(mx)r=0r=0
$ and $m(rx)=(mr)x=0.$
Thus, $\mathtt{Ann}_{R}(M)$ is a hyperideal of $R$.
\end{proof}

\subsection{Primitivity}

A proper hyperideal of a hyperring $R$ is called \emph{primitive} if it is the annihilator of a simple $R$-hypermodule. We shall denote the set of all primitive hyperideals of $R$ by $\mathtt{Prim}(R)$. A hyperring $R$ is said to be \emph{primitive} if $\{0\}$ is a primitive hyperideal of $R.$ The next two propositions show some implications between maximal, prime, and primitive hyperideals.

\begin{proposition}\label{prtpr}
Every primitive hyperideal is a prime hyperideal.
\end{proposition}

\begin{proof}
Suppose $\mathfrak{p}=\mathtt{Ann}_{R}(M)$ for some simple $R$-hypermodule $M$, and  $\mathfrak{b}$ is a hyperideal of $R$ such  that $M\mathfrak{b} \neq 0,$ that is, $\mathfrak{b} \not\subseteq \mathfrak{p}.$ Since $M$ is simple, we must have $M\mathfrak{b} = M.$
If $\mathfrak{a}$ is a nonzero hyperideal of $R$, then
\begin{equation}\label{mba}
M(\mathfrak{b}\mathfrak{a}) = (M\mathfrak{b})\mathfrak{a} = M\mathfrak{a} = M,
\end{equation} which implies $M\mathfrak{a} \neq 0,$ that is, $\mathfrak{a} \not\subseteq \mathfrak{p}.$ Therefore, from (\ref{mba}) it follows that  $\mathfrak{b}\mathfrak{a} \not\subseteq \mathfrak{p}.$ 
\end{proof}

\begin{proposition}\label{maxi}
Every maximal hyperideal of a unital hyperring is a primitive hyperideal.
\end{proposition}

\begin{proof}
Suppose $\mathfrak{a}$ is maximal hyperideal of a hyperring $R.$ Then by Lemma \ref{max}, $\mathfrak{a}$ is contained in a maximal right hyperideal $\mathfrak{b}$ of $R$  and $\mathfrak{a} \subseteq \mathtt{Ann}(R/\mathfrak{b}).$ Since $\mathfrak{a}$ is a maximal hyperideal of $R,$ we must have that $\mathfrak{a}= \mathtt{Ann}(R/\mathfrak{b}),$ and thus $\mathfrak{a}$ is the annihilator of a simple $R$-hypermodule $R/\mathfrak{b}$. 
\end{proof}

From the definition above, it is clear that a hyperring $R$ is primitive if and only if the zero hyperideal of $R$ is a primitive hyperideal. This equivalence can further be generalized for an arbitrary primitive hyperideal of $R$.

\begin{proposition}
A hyperideal $\mathfrak{p}$ of a hyperring $R$ is  primitive if and only if $R/ \mathfrak{p}$ is a primitive hyperring.
\end{proposition}

\begin{proof}
Suppose $\mathfrak{p}$ is primitive hyperideal of  $R$ and let $M$ be a simple $R$-hypermodule such that $\mathfrak{p} = \mathtt{Ann}(M)$. If we define 
$m(\mathfrak{p}+r)= mr \mbox{ for all $r \in r, m \in M$}$, then the additive canonical hypergroup of $M$ is also a simple $R/\mathfrak{p}$-hypermodule. On the other hand, since $\mathtt{Ann}(M) \subseteq \mathfrak{p},$ we have that $M$ is a faithful $R/\mathfrak{p}$-hypermodule. Conversely, suppose  that $N$ is a faithful simple $R/\mathfrak{p}$-hypermodule and for all $r \in R, n \in N$, define
$nr = n(\mathfrak{p}+r).$ Then the additive canonical hypergroup of $N$ becomes a simple $R$-hypermodule with $\mathtt{Ann}(N) = 
\mathfrak{p}.$
\end{proof}

Primitive hyperideals are also related to right maximal hyperideals which we will see in the next proposition. But for that we need the following

\begin{lemma}\label{lemsimple}
Let $R$ be a hyperring. An $R$-hypermodule $M$ is simple if and only if $M$ is isomorphic to $R/\mathfrak{m}$ for some maximal right hyperideal $\mathfrak{m}$ of $R.$  
\end{lemma}

\begin{proof}
Let $M$ be a simple $R$-hypermodule. Choose $0 \neq m \in M.$ Then $mR = M$ and hence $\psi: R \rightarrow M,$ defined by $\psi(r) = mr,$ is a surjective $R$-hypermodule homomorphism. Its kernel $\mathfrak{m}$ is a right hyperideal of $R$ and by Proposition \ref{fth}, we have $R/\mathfrak{m} \cong M.$ To show that $\mathfrak{m}$ is maximal, let $\mathfrak{b}$ be a right hyperideal of $R$ such that $\mathfrak{m} \subseteq \mathfrak{b} \subseteq R.$ Then $\mathfrak{b}/\mathfrak{a}$ is a subhypermodule of $R/\mathfrak{m}.$ Now since $R/\mathfrak{m}$ is isomorphic to $M$ and since $M$ is simple, we must have either $\mathfrak{b}/\mathfrak{m}=0$ or $\mathfrak{b}/\mathfrak{a} = R/\mathfrak{a},$ and thus, either $\mathfrak{b} = \mathfrak{a}$ or $\mathfrak{b} = R,$ which implies $\mathfrak{m}$ is maximal.
Conversely, let $\mathfrak{m}$ be a maximal hyperideal of $R$ and consider a subhypermodule $N$ of $R/\mathfrak{m}.$ It is easy to see that $\mathfrak{b}=\{r \in R \mid \mathfrak{m} + r \in N\}$ is a right hyperideal of $R$ containing $\mathfrak{m}$. Thus $\mathfrak{b} = \mathfrak{a}$ or $\mathfrak{b} = R,$ giving that $N = 0$ or $N = R/\mathfrak{m}.$ Thus $R/\mathfrak{m}$ is a simple $R$-hypermodule.
\end{proof}

\begin{proposition}
If $\mathfrak{p}$ is a primitive hyperideal of a hyperring $R$ then there exists a maximal right hyperideal $\mathfrak{m}$ of $R$ such that 
\begin{equation}\label{prr}
\mathfrak{p} = \{r \in R \mid Rr \subseteq \mathfrak{m}\}.
\end{equation}
Conversely, if $\mathfrak{m}$ is a maximal right hyperideal of $R$ and if $R^2\nsubseteq \mathfrak{m}$, then the hyperideal $\mathfrak{p}$ defined in (\ref{prr}) is primitive.
\end{proposition}

\begin{proof}
If $\mathfrak{p} = \mathtt{Ann}_R(M)$, for some simple $R$-hypermodule $M$, then by Lemma \ref{lemsimple}, there exists a maximal right hyperideal $\mathfrak{m}$ of $R$ such that $M \cong R/\mathfrak{m}.$ This implies $\mathfrak{p} = \mathtt{Ann}_{R}(R/\mathfrak{m})$ and hence satisfies condition (\ref{prr}).
Conversely, if we assume that $\mathfrak{m}$ is a maximal right hyperideal of $R,$ then again by Lemma \ref{lemsimple}, $R/\mathfrak{m}$ is a simple $R$-hypermodule, and therefore,  $\mathtt{Ann}_{R}(R/\mathfrak{m}) = \mathfrak{p}$, a primitive hyperideal of $R.$
\end{proof}

\begin{corollary}
Every maximal right hyperideal of a unital hyperring contains a primitive hyperideal.
\end{corollary}

\section{Hyperstructure spaces}

We shall introduce Jacobson
topology in $\mathtt{Prim}(R),$ the set of primitive hyperideals of a hyperring $R,$  by defining a closure operator for the subsets of $\mathtt{Prim}(R)$. Once we have a closure operator, closed sets are defined as sets which are invariant under this closure operator. Suppose $S$ is a subset of $\mathtt{Prim}(R)$. Set $\mathcal{K}_S=\bigcap_{\mathfrak{q}\in S}\mathfrak{q}.$ We define the closure of the set $S$ as 
\begin{equation}\label{clop}
\mathtt{Cl}(S)=\left\{ \mathfrak{p}\in \mathtt{Prim}(R) \mid \mathfrak{p}\supseteq \mathcal{K}_S \right\}.
\end{equation}

If $S=\{\mathfrak{s}\}$, we will write $\mathtt{Cl}(\{\mathfrak{s}\})$ as $\mathtt{Cl}(\mathfrak{s})$. We wish to verify that the closure operation defined  in (\ref{clop}) satisfies Kuratowski's closure conditions, and that is done in the following.

\begin{proposition}\label{ztp}
The sets $\{\mathtt{Cl}(S)\}_{S\subseteq  \mathtt{Prim}(R)}$ satisfy the following conditions for all subsets $S$ and $T$ of the hyperstructure space $\mathtt{Prim}(R)$:
\begin{enumerate} [\upshape(i)]
		
\item\label{clee} $\mathtt{Cl}(\emptyset)=\emptyset;$
		
\item\label{clxx} $\mathtt{Cl}(S)\supseteq S;$
		
\item\label{clclx} $\mathtt{Cl}(\mathtt{Cl}(S))=\mathtt{Cl}(S);$
		
\item\label{clxy} $ \mathtt{Cl}(S\cup T)=\mathtt{Cl}(S)\cup \mathtt{Cl}(T).$ 
\end{enumerate}
\end{proposition}

\begin{proof}
The proofs of (\ref{clee})-(\ref{clclx}) are straightforward, whereas for (\ref{clxy}), it is easy to see that $ \mathtt{Cl}(S\cup T)\supseteq\mathtt{Cl}(S)\cup \mathtt{Cl}(T).$ To obtain the the other inclusion, let $\mathfrak{p}\in \mathtt{Cl}(S\cup T).$ Then
$$\mathfrak{p}\supseteq \mathcal{K}_{S\cup T}=\mathcal{K}_S \cap \mathcal{K}_T.$$
Since $\mathcal{K}_S$ and $\mathcal{K}_T$ are hyperideals of  the hyperring $R$, it follows that 
$$\mathcal{K}_S\mathcal{K}_T\subseteq \mathcal{K}_S \cap \mathcal{K}_T\subseteq \mathfrak{p}.$$
Since by Proposition \ref{prtpr}, $\mathfrak{p}$ is prime, either $\mathcal{K}_S\subseteq \mathfrak{p}$ or $\mathcal{K}_T\subseteq \mathfrak{p}.$ This means either $\mathfrak{p}\in \mathtt{Cl}(S)$ or $\mathfrak{p}\in \mathtt{Cl}(T)$. Thus $ \mathtt{Cl}(S\cup T)\subseteq\mathtt{Cl}(S)\cup \mathtt{Cl}(T).$ 
\end{proof}

The set $\mathtt{Prim} (R)$ of primitive hyperideals of a hyperring $R$ topologized (the Jacobson topology) by the 
closure operator defined in (\ref{clop}) is called the \emph{hyperstructure space} of the hyperring $R$. If $S$ is a subset of a hyperring $R$, then $$\mathcal{O}(S)=\{\mathfrak{p}\in \mathtt{Prim}(R)\mid \mathfrak{p}\nsupseteq \mathcal{K}_S\}$$ is a typical open subset of this topology.
It is evident from (\ref{clop}) that if $\mathfrak{p}\neq \mathfrak{p}'$ for any two $\mathfrak{p}, \mathfrak{p}'\in \mathtt{Prim}(R)$, then $\mathtt{Cl}(\mathfrak{p})\neq \mathtt{Cl}(\mathfrak{p}').$ Thus we have the following.

\begin{proposition}\label{t0a}
Every hyperstructure space $\mathtt{Prim}(R)$ is a $T_0$-space. 
\end{proposition} 

Using the finite intersection property, we can obtain compactness of a hyperstructure space.

\begin{theorem}\label{csb}
If $R$ is a unital hyperring  then the hyperstructure space $\mathtt{Prim}(R)$ is compact. 
\end{theorem}
\begin{proof}
Let  $\{C_{ \lambda}\}_{\lambda \in \Lambda}$ be a family of  closed sets of a hyperstructure space $\mathtt{Prim}(R)$   such that $\bigcap_{\lambda\in \Lambda}C_{ \lambda}=\emptyset.$ Then a primitive hyperideal $\mathfrak{p}\in \bigcap_{\lambda\in \Lambda}C_{ \lambda}$ if and only if  $ \mathfrak{p}\supseteq \sum_{\lambda \in \Lambda}\mathcal{K}_{ C_{\lambda}}.$ Since $\bigcap_{\lambda\in \Lambda}C_{ \lambda}=\emptyset,$ we must have  $ \sum_{\lambda \in \Lambda}\mathcal{K}_{ C_{\lambda}}=R.$ Then, in particular, we obtain $1=\sum_{i=1}^n\mathcal{K}_{ C_{\lambda_i}}$ for a suitable finite subset $\{\lambda_1, \ldots, \lambda_n\}$ of $\Lambda$. This in turn implies $\bigcap_{ i\,=1}^{ n}C_{ \lambda_i}=\emptyset,$ and hence $\mathtt{Prim}(R)$ is compact.  
\end{proof} 

Recall that a nonempty closed subset $C$ of a topological space $X$ is \emph{irreducible} if $C\neq C_1\cup C_2$ for any two proper closed subsets  $C_1, C_2$ of $C$. A maximal irreducible subset of a topological space $X$ is called an
\emph{irreducible component} of $X.$ A point $x$ in a closed subset $C$ is called a \emph{generic point} of $C$ if $C = \mathtt{Cl}(x).$  

\begin{lemma}\label{lemprime}
$\{\mathtt{Cl}(\mathfrak{p})\}_{\mathfrak{p}\in \mathtt{Prim}(R)}$ are the only irreducible closed subsets of a hyperstructure space $\mathtt{Prim}(R).$  
\end{lemma}

\begin{proof} 
Since $\{\mathfrak{p}\}$ is irreducible, so is  $\mathtt{Cl}(\mathfrak{p}).$  Suppose $\mathtt{Cl}(\mathfrak{a})$ is an irreducible closed subset of $\mathtt{Prim}(R)$ and $\mathfrak{a}\notin \mathtt{Prim}(R).$ This implies there exist hyperideals $\mathfrak{b}$ and $\mathfrak{c}$ of $R$ such that  $\mathfrak{b}\nsubseteq \mathfrak{a}$ and $\mathfrak{c}\nsubseteq \mathfrak{a}$, but $\mathfrak{b}\mathfrak{c}\subseteq \mathfrak{a}$. Then   
$$\mathtt{Cl}(\langle \mathfrak{a}, \mathfrak{b}\rangle)\cup \mathtt{Cl}(\langle \mathfrak{a},\mathfrak{c}\rangle)=\mathtt{Cl}(\langle \mathfrak{a},  \mathfrak{b}\mathfrak{c}\rangle)=\mathtt{Cl}(\mathfrak{a}).$$
But $\mathtt{Cl}(\langle \mathfrak{a},  \mathfrak{b}\rangle)\neq \mathtt{Cl}(\mathfrak{a})$ and $\mathtt{Cl}(\langle \mathfrak{a},  \mathfrak{c}\rangle)\neq \mathtt{Cl}(\mathfrak{a}),$ and hence $\mathtt{Cl}(\mathfrak{a})$ is not irreducible.
\end{proof}

\begin{proposition}
Every irreducible closed subset of $\mathtt{Prim}(R)$ has a unique generic point.
\end{proposition}

\begin{proof}
The existence of a generic point follows from Lemma \ref{lemprime}, and the uniqueness of such a point follows from Proposition \ref{t0a}. 
\end{proof} 

The irreducible components of a hyperstructure space can be characterised in terms of minimal primitive hyperideals, and we have that in the following result.

\begin{proposition}\label{thmirre}
The irreducible components of a hyperstructure space $\mathtt{Prim}(R)$ are the closed sets $\mathtt{Cl}(\mathfrak{p})$, where $\mathfrak{p}$ is a minimal primitive hyperideal of $R$. 
\end{proposition}

\begin{proof}
If $\mathfrak{p}$ is a minimal primitive hyperideal, then by Lemma \ref{lemprime}, $\mathtt{Cl}(\mathfrak{p})$ is irreducible. If $\mathtt{Cl}(\mathfrak{p})$ is not a maximal irreducible subset of $\mathtt{Prim}(S)$, then there exists a maximal irreducible subset $\mathtt{Cl}(\mathfrak{p}')$ with $\mathfrak{p}'\in  \mathtt{Prim}(S)$ such that $\mathtt{Cl}(\mathfrak{p})\subsetneq \mathtt{Cl}(\mathfrak{p}')$. This implies that $\mathfrak{p}\in \mathtt{Cl}(\mathfrak{p}')$ and hence $\mathfrak{p}'\subsetneq \mathfrak{p}$, contradicting the minimality property of $\mathfrak{p}$.
\end{proof}

Recall that a hyperring is called \emph{Noetherian} if it satisfies the ascending
chain condition, whereas a topological space $X$ is called \emph{Noetherian} if the descending chain condition holds for closed subsets of $X.$ A relation between these two notions is shown in the following.

\begin{proposition}\label{fwn}
If a hyperring $R$ is Noetherian, then $\mathtt{Prim}(R)$ is a Noetherian hyperstructure space.  
\end{proposition}

\begin{proof}
It suffices to show that a collection of closed sets in $\mathtt{Prim}(R)$ satisfy the descending chain condition. Let $\mathtt{Cl}(\mathfrak{a}_1)\supseteq \mathtt{Cl}(\mathfrak{a}_2)\supseteq \cdots$
be a descending chain of closed sets in $\mathtt{Prim}(R)$. Then, $\mathfrak{a}_1\subseteq \mathfrak{a}_2\subseteq \cdots$ is an ascending chain of hyperideals in $R.$ Since the hyperring $R$ is Noetherian, the chain stabilizes at some $n \in \mathbb{N}.$ Hence, $\mathtt{Cl}(\mathfrak{a}_n) = \mathtt{Cl}(\mathfrak{a}_{n+k})$ for any $k.$ Thus $\mathtt{Prim}(R)$ is Noetherian.
\end{proof}
\medskip

\begin{corollary}
The set of minimal primitive hyperideals in a Noetherian hyperring is finite.
\end{corollary}

\begin{proof}
By Proposition \ref{fwn}, $\mathtt{Prim}(R)$ is Noetherian, thus $\mathtt{Prim}(R)$ has a finitely many irreducible
components. By Proposition \ref{thmirre}, every irreducible closed subset of $\mathtt{Prim}(R)$
is of the form $\mathtt{Cl}(\mathfrak{p}),$ where $\mathfrak{p}$ is a minimal primitive hyperideal. Thus $\mathtt{Cl}(\mathfrak{p})$ is an irreducible component if and only if $\mathfrak{p}$ is a minimal primitive hyperideal. Hence, $R$ has only finitely many minimal primitive hyperideals.
\end{proof}

In general, a hyperstructure space is not $T_1$. However, with added restriction we can characterize such spaces. 

\begin{theorem}
An hyperstructure space $\mathtt{Prim}(R)$ is a $T_1$-hyperstructure space if and only if $\mathtt{Prim}(R)$ coincides with the set $\mathtt{Max}(R)$ of maximal hyperideals of $R$. 
\end{theorem}  

\begin{proof} 
By Proposition \ref{maxi}, $\mathtt{Max}(R)\subseteq \mathtt{Prim}(R).$ So, it is sufficient to prove the result for the other inclusion. Let $\mathfrak{a}\in \mathtt{Prim}(R)$. Then $\mathfrak{a}\in\mathtt{Cl}(\mathfrak{a})$. Let $\mathfrak{m}$ be a maximal hyperideal with $\mathfrak{a}\subseteq\mathfrak{m}$. Then   $$\mathfrak{m}\in 	\mathtt{Cl}(\mathfrak{a})= \{\mathfrak{a}\},$$where the equality follows from $\mathtt{Prim}(R)$ being a $T_{ 1}$-space. Therefore $\mathfrak{m}=\mathfrak{a}$, showing that $\mathtt{Prim}(R)\subseteq \mathtt{Max}(R)$.  	
Conversely, in $\mathtt{Max}(R)$, $\mathtt{Cl}(\mathfrak{m})=\{\mathfrak{m}\}$ for every maximal hyperideal $\mathfrak{m}$, so that $\mathfrak{m}\in \mathtt{Cl}(\mathfrak{m})$, showing that the hyperstructure space is $T_{ 1}$.
\end{proof} 

A strong hyperring homomorphism induces a continuous map between corresponding hyperstructure spaces. We now study this continuity and homeomorphisms between such spaces.

\begin{proposition}\label{conmap}
Suppose $\phi\colon R\to R'$ is a strong hyperring  homomorphism and define the map $\phi_*\colon  \mathtt{Prim}(R')\to \mathtt{Prim}(R)$ by  $\phi_*(\mathfrak{p})=\phi\inv(\mathfrak{p})$, where $\mathfrak{p}\in\mathtt{Prim}(R').$ Then $\phi_*$ is a continuous map.
\end{proposition}

\begin{proof}
To show $\phi_*$ is continuous, we first show that $\phi\inv(\mathfrak{p})\in \mathtt{Prim}(R),$ whenever $\mathfrak{p}\in \mathtt{Prim}(R')$. Note that $\phi\inv(\mathfrak{p})$ is a hyperideal of $R$. Suppose $\mathfrak{p}=\mathtt{Ann}_{R'}(M)$ for some simple $R'$-hypermodule. Then by the ``change of hyperrings'' property of hypermodules, $\phi\inv(\mathfrak{p})$ is the annihilator of the simple $R'$-hypermodule $M$ obtained by defining $sm=\phi(s)m$. Therefore $\phi\inv(\mathfrak{p})\in \mathtt{Prim}(R)$. Now consider a closed subset $\mathtt{Cl}(\mathfrak{a})$ of  $\mathtt{Prim}(R).$ Then for any $\mathfrak{q}\in \mathtt{Prim}(R'),$ we have the following sequence of equivalent statements:
$$
\mathfrak{q}\in \phi_*\inv (\mathtt{Cl}(\mathfrak{a}))\Leftrightarrow \phi\inv(\mathfrak{q})\in \mathtt{Cl}(\mathfrak{a})\Leftrightarrow \mathfrak{a}\subseteq \phi\inv(\mathfrak{q})\Leftrightarrow \mathfrak{q}\in\mathtt{Cl}(\langle \phi(\mathfrak{a})\rangle),
$$
and that proves the desired continuity of $\phi_*$.
\end{proof}

\begin{proposition}
If $\mathfrak{a}$ is a hyperideal of the hyperring $R,$ then $\mathtt{Cl}(\mathfrak{a})$ is  homeomorphic  to the hyperstructure space $\mathtt{Prim}(R/\mathfrak{a}).$
\end{proposition}

\begin{proof}
We shall in fact prove more, viz, if $\phi\colon R\to R'$ is a strong hyperring homomorphism and if $\phi$ is  surjective, then the hyperstructure space $\mathtt{Prim}(R')$ is homeomorphic to the closed subset $\mathtt{Cl}(\mathtt{ker}(\phi))$ of the hyperstructure space $\mathtt{Prim}(R).$ The desired result will then follow by taking the quotient map $R\to R/\mathfrak{a}.$

Since  $\mathfrak{o}\subseteq \mathfrak{b}$ for all $\mathfrak{b}\in \mathtt{Prim}(R'),$ we have     
$\mathtt{ker}(\phi)\subseteq \phi\inv(\mathfrak{b}),$ or, in other words $f^*(\mathfrak{b})\in \mathtt{Cl}(\mathtt{ker}(\phi)).$ This implies that $\mathtt{im}(\phi^*)=\mathtt{Cl}(\mathtt{ker}(\phi)).$  
Since for all $\mathfrak{b}\in \mathtt{Prim}(R'),$ $\phi(\phi^*(\mathfrak{b}))=\phi(\phi\inv(\mathfrak{b}))=\mathfrak{b},$ the map $\phi^*$ is injective. To show that $\phi^*$ is a closed map, first we observe that for any closed subset  $\mathtt{Cl}(\mathfrak{a})$ of  $\mathtt{Prim}(R')$, we have: $$\phi^*(\mathtt{Cl}(\mathfrak{a}))=  \phi\inv(\mathtt{Cl}(\mathfrak{a}))=\phi\inv\{ \mathfrak{i}'\in \mathtt{Prim}(R')\mid \mathfrak{a}\subseteq   \mathfrak{i}'\}=\mathtt{Cl}(\phi\inv(\mathfrak{a})).$$

Now if $C$ is a closed subset of $\mathtt{Prim}(R')$ and $C=\mathtt{Cl}(\mathfrak{a}),$ then $\phi^*(C)=\phi\inv \left(\mathtt{Cl}(\mathfrak{a})\right)= \mathtt{Cl}(\phi\inv(\mathfrak{a})),$ a closed subset of  $\mathtt{Prim}(R).$ Since by Proposition \ref{conmap}, $\phi^*$ is   continuous, we have the desired claim.	
\end{proof}

\begin{corollary}
The hyperstructure spaces  $\mathtt{Prim}(R)$ and $\mathtt{Prim}(R)/\sqrt{\mathfrak{o}}$ are  homeomorphic, where $\sqrt{\mathfrak{o}}$ is the nil radical of $R$. 
\end{corollary}

\begin{proposition}
Let $\phi^*$ be as in Proposition \ref{conmap}. Then  $\phi^*(\mathtt{Prim}(R'))$ is dense in $\mathtt{Prim}(R)$ if and only if $\mathtt{ker}(\phi)\subseteq \sqrt{\mathfrak{o}}.$ 
\end{proposition}

\begin{proof}
We first show that $\mathtt{Cl}(\phi^*(\mathtt{Cl}(\mathfrak{b})))=\mathtt{Cl}(\phi\inv(\mathfrak{b})),$ for all hyperideals $\mathfrak{b}$ of $R'.$ To this end, let $\mathfrak{s}\in \phi^*(\mathtt{Cl}(\mathfrak{b})).$ This implies $\phi(\mathfrak{s})\in \mathtt{Cl}(\mathfrak{b}),$ which means $\mathfrak{b}\subseteq \phi(\mathfrak{s}).$ In other words, $\mathfrak{s}\in \mathtt{Cl}(\phi\inv(\mathfrak{b})).$ The other inclusion follows from the fact that $\phi\inv(\mathtt{Cl}(\mathfrak{b}))=\mathtt{Cl}(\phi\inv(\mathfrak{b})).$ Since $$\mathtt{Cl}(\phi^*(\mathtt{Prim}(R')))=\phi^*(\mathtt{Cl}(\mathfrak{o}))=\mathtt{Cl}(\phi\inv(\mathfrak{o}))=\mathtt{Cl}(\mathtt{ker}(\phi)),$$ we see that $\mathtt{Cl}(\mathtt{ker}(\phi))$ is equal to $\mathtt{Prim}(R))$ if and only if $\mathtt{ker}(\phi)\subseteq \sqrt{\mathfrak{o}}.$ 
\end{proof}



\end{document}